\documentclass[reqno,12pt]{amsart}
\usepackage{amsmath,amsthm,enumerate,graphics,pstricks,amsfonts,latexsym,amsopn,verbatim,amscd,amssymb}
\usepackage{color}
\usepackage{psfrag}
\usepackage[all]{xy}

\theoremstyle{plain}
\newtheorem{thm}{Theorem}[section]

\newtheorem{lem}[thm]{Lemma}

\newtheorem{prop}[thm]{Proposition}

\theoremstyle{definition}
\newtheorem{rem}[thm]{Remark}

\begin{document}

\title[Genus of the commuting graphs]{On the genus of the commuting graphs of finite non-abelian groups}

\author[A. K. Das  and  D. Nongsiang]{Ashish Kumar Das*}

\address{A. K. Das, Department of Mathematics, North-Eastern Hill University,
Permanent Campus, Shillong-793022, Meghalaya, India.}

\email{akdasnehu@gmail.com}

\author[]{  Deiborlang Nongsiang}
\address{D. Nongsiang, Department of Mathematics, North-Eastern Hill University,
Permanent Campus, Shillong-793022, Meghalaya, India.}

\email{ndeiborlang@yahoo.com}

\begin{abstract}
The commuting graph of a non-abelian group is a simple graph in which the vertices are the non-central elements of the group, and two distinct vertices are adjacent if and only if they commute. In this paper, we classify (up to isomorphism) all finite non-abelian groups whose commuting graphs are acyclic, planar or toroidal. We also derive  explicit formulas for the genus of the commuting graphs of some well-known class of finite non-abelian groups, and show that, every collection of finite non-abelian groups whose commuting graphs have the same genus is finite.
\end{abstract}

\thanks{*Corresponding author}

\maketitle

\section{Introduction} \label{S:intro}
Let G be a non-abelian group and $Z(G)$ be its center. The {\em commuting graph} of $G$, denoted by $\Gamma_c (G)$, is a simple undirected graph in which the vertex set is $G \setminus Z(G)$,  and two vertices $x$ and $y$ are adjacent if and only if $xy=yx$. This graph is precisely the complement of the non-commuting graph of a group considered in \cite{aam} and \cite{mog}.  The origin of this notion lies in a seminal paper by R. Brauer and K. A. Fowler \cite{bf} who were concerned primarily with the classification of the finite simple groups. However, the ever-increasing popularity of the topic is often attributed to a question, posed in 1975 by Paul Erd\"{o}s and answered   affirmatively  by B. H. Neumann \cite{neu}, asking whether or not a  non-commuting graph having no infinite complete subgraph possesses a finite bound on the cardinality of  its  complete subgraphs. In recent years, the  commuting graphs of groups has become a topic of research for many mathematicians (see, for example, \cite{bbhr}, \cite{ira1}). In \cite{ira2}, it was conjectured that the commuting graph of a finite group is either disconnected or has diameter bounded above by a constant independent of the group G. This conjecture was well-supported in \cite{par} and \cite{ss}. However, in \cite{gp}, it is shown  that, for all positive integers $d$, there exists a finite special $2$-group G such that the commuting graph of G has diameter greater than $d$. But in \cite{mor}, it is proved that for finite groups with trivial center the conjecture made in \cite{ira2} holds good.  The concept of commuting graphs of groups (taking, as the vertices, the non-trivial elements of the group in place of non-central elements) has also been recently used in  \cite{rap} to show that finite quotients of the multiplicative group of a finite dimensional division algebra are solvable. There is also a ring  theoretic version  of commuting graphs  (see, for example, \cite{akb1}, \cite{akb2}). 

Most of the works cited above on  commuting graphs of groups deal with connectedness, diameter and some algebraic aspects of the graph. Also, some of the results for the non-commuting graphs of groups have their obvious analogues for the commuting counterparts, the commuting and non-commuting graphs being complements of each other. In the present  paper, however, we deal with a topological aspect, namely, the genus of the commuting graphs of finite non-abelian groups,  and on this count the commuting and the non-commuting graphs are independent of each other. Here we show that every collection of finite non-abelian groups whose commuting graphs have the same genus is finite.   One of the sections in this paper is  devoted entirely to the computation of the genus of the commuting graphs of some well-known families of finite non-abelian groups.  The primary objective of this paper is, of course, to determine, up to isomorphism, all finite non-abelian groups whose commuting graphs are planar or toroidal, that is, can be drawn on the surface of a sphere or of a torus (without any crossing of edges).  We, however, begin by classifying  all non-abelian groups whose commuting graphs have no triangles, which, in fact, turns out to be equivalent to saying that the corresponding non-commuting graphs are planar. It may be mentioned here that the motivation for this paper comes from \cite{das}, \cite{mai}, \cite{wang}, \cite{wic1} and \cite{wic2}, where similar problems for certain graphs associated to finite rings have been addressed.

\section{Some prerequisites}\label{S:pre}
In this section, we recall certain graph theoretic  terminologies (see, for example, \cite{west} and \cite{whit}) and some well-known results which have been used extensively in the
forthcoming sections. Note that all graphs considered in this and the following sections are  simple graphs, that is, graphs   without loops or multiple edges.

Let $\Gamma$ be a graph with vertex set $V(\Gamma)$ and edge set $E(\Gamma)$.  Let $x,y \in V(\Gamma)$. Then  $x$ and $y$ are said to be {\it adjacent} if $x\neq y$ and there is an edge $x-y$ in $E(\Gamma)$ joining $x$ and $y$.   A path between $x$ and $y$ is a sequence  of adjacent vertices, often written as $x - x_1 - x_2 - \dots -x_n - y$, where the vertices $x, x_1, x_2, \dots, x_n, y$ are all distinct (except, possibly, $x$ and $y$).   $\Gamma$ is said to be {\it connected} if there is a path between every pair of distinct vertices in $\Gamma$.  A path between $x$ and $y$ is called a {\it cycle} if $x=y$. The number of edges in a path or a cycle, is called its {\it length}.  A cycle of length $n$ is called an $n$-cycle, and a $3$-cycle is also called a triangle. The {\it girth} of    $\Gamma$ is the minimum of the lengths of all cycles in     $\Gamma$, and is denoted by ${\rm girth}(\Gamma)$. If       $\Gamma$ is {\it acyclic}, that is, if $\Gamma$ has no cycles, then we write ${\rm girth}(\Gamma)=\infty$. 

A graph $G$ is said to be {\it complete} if there is an edge between every pair of distinct vertices in $G$. We denote the complete graph with $n$ vertices by $K_n$. A {\it bipartite graph} is the one whose vertex set can
be partitioned into two disjoint parts in such a way that the two end vertices of every edge lie in different parts. Among the bipartite graphs, the {\it complete bipartite graph} is the one in which two vertices are
adjacent if and only if they lie in different parts. The complete bipartite graph, with parts of size $m$ and $n$, is denoted by $K_{m,n}$.

A subset of the vertex set of a graph $\Gamma$ is called a \textit{clique} of $\Gamma$ if it consists entirely of pairwise adjacent vertices. The least upper bound of the sizes of all the cliques of $G$ is called the \textit{clique number} of $\Gamma$, and is denoted by $\omega (\Gamma)$. The  \textit{chromatic number}  of a graph $\Gamma$, written $\chi(\Gamma)$, is the minimum number of colors needed to label the vertices so that adjacent vertices receive different colors. Clearly, $\omega (\Gamma) \leq \chi(\Gamma)$.
 
Given a graph $\Gamma$, let $U$ be a nonempty subset of   $V(\Gamma)$. Then  the \textit{induced subgraph} of $\Gamma$ on $U$  is defined to be the graph $\Gamma [U]$  in which the vertex set  is $U$ and  the edge set consists precisely of those edges in $\Gamma$ whose endpoints lie in $U$.  If  $\{\Gamma_{\alpha}\}_{\alpha \in \Lambda}$ is a family of subgraphs of a graph   $\Gamma$, then the union             $\underset{\alpha \in \Lambda} \cup  \Gamma_{\alpha}$ denotes the subgraph of $\Gamma$ whose vertex set is $\underset{\alpha \in \Lambda} \cup V(\Gamma_{\alpha})$ and the edge set is   $\underset{\alpha \in \Lambda} \cup E(\Gamma_{\alpha})$. Further, given a graph $\Gamma$, its complement is defined to be the graph in which the vertex set is the same as the one in $\Gamma$ and  two distinct vertices  are adjacent if and only if they are not adjacent vertices in $\Gamma$.

The \textit{genus} of a graph $\Gamma$, denoted by          $\gamma(\Gamma)$, is the smallest non-negative integer $n$ such that the graph can be embedded on the surface obtained by attaching $n$ handles to a sphere. Clearly, if $\tilde{\Gamma}$ is a subgraph of $\Gamma$, then $\gamma(\tilde{\Gamma}) \leq \gamma(\Gamma)$. Graphs having genus zero are called \textit{planar} graphs, while those having genus one are called \textit{toroidal} graphs.

A {\em block} of a graph $\Gamma$ is a connected subgraph $B$ of $\Gamma$ that is maximal with respect to the property that removal of a single vertex (and the incident edges) from $B$ does not make it disconnected, that is, the graph $B\setminus \{v\}$ is connected for all $v \in V(B)$. Given a graph $\Gamma$, there is  a unique finite collection $\mathfrak{B}$ of blocks   of $\Gamma$, such that $ \Gamma  = \underset{B \in \mathfrak{B}} \cup B$. The collection 
$\mathfrak{B}$ is called the {\em block decomposition} of $\Gamma$.  In \cite[Corollary 1]{bhky}, it has been proved  that the genus of a graph is the sum of the genera of its blocks. Thus, it follows that 

\begin{lem}\label{lem1}
If a graph $\Gamma$ has two disjoint subgraphs $\Gamma_1$ and $\Gamma_2$ such that $\Gamma_1 \cong K_m$ and $\Gamma_2 \cong K_n$ for some positive integers $m$ and $n$, then   $\gamma(\Gamma ) \geq  \gamma(K_n) + \gamma(K_m)$.
\end{lem}
 
 We conclude the section with the following two useful results.

\begin{lem}[\cite{whit}, Theorem 6-38] \label{kn}
If $n\geq 3$, then $$\gamma(K_{n})=\left\lceil \frac{(n-3)(n-4)}{12}\right
\rceil.$$
\end{lem}

\begin{lem}[\cite{whit}, Theorem 6-37] \label{knm}
If $m,n\geq 2$, then $$\gamma(K_{m,n})=\left\lceil \frac{(m-2)(n-2)}{4}
\right\rceil.$$
\end{lem}

\section{Some basic results} 
In this section we derive some results concerning the genus of the commuting graphs of finite groups which are not only of interest in their own right but also used extensively in the forthcoming sections.

In the study of the genus of a graph, the cycles in the graph play a crucial role. Therefore, determining whether or not the graph is acyclic can be considered as the first step in this direction. Even otherwise, whether or not a graph associated to a group has a triangle is a topic of substantial interest (see, for example \cite{fang}). Keeping this is mind, we begin the section with the following result which, in view of  \cite[Proposition 2.3]{aam}, also says that the commuting graph of a non-abelian group is acyclic if and only if its complement (that is, the non-commuting graph of the group) is planar. 
\begin{prop} \label{no3cyc} 
Let $G$ be a non-abelian group. Then, $\Gamma_c (G)$ has no 3-cycle if and only if $G$ is isomorphic to the symmetric group  $S_3$, the quaternion group $Q_8$,  or the dihedral group  $D_8$.
\end{prop}

\begin{proof}
If $G$ is isomorphic to   $S_3$, $Q_8$  or  $D_8$, then it is easy to see that $\Gamma_c (G)$ has no 3-cycle; in fact, $\Gamma_c (G)$ is acyclic.

Conversely, suppose that $\Gamma_c (G)$ has no 3-cycle. Then, $|Z(G)| \leq 2$; otherwise, for all $x\in G \setminus Z(G)$, the induced subgraph $\Gamma_c (G)[xZ(G)] \cong K_{|Z(G)|}$ would contain a 3-cycle.

\noindent \textit{Case 1.} \ $|Z(G)| = 1$.

In this case, every element of $G$ has order 2 or 3; otherwise $\{x, x^2, x^3 \}$ would form a 3-cycle in $\Gamma_c (G)$ for all $x\in G$ with $o(x) >3$.  Therefore, $G$ is a group of exponent dividing 6. Let $H$ be a finitely generated subgroup of $G$.  Then, $H$ is a finite group (see \cite[Sections 18.2, 18.4]{mhall}), and so, we have  $|H|= 2^m 3^n$ for some non-negative integers $m$ and $n$. If $m\geq 2$, then the Sylow 2-subgroup $H_2$ of $H$ exists, and is elementary abelian, which implies that  the induced subgraph $\Gamma_c (G)[H_2 \setminus \{1\}] \cong K_{2^m -1}$ contains a 3-cycle.  Therefore, we have $m\leq 1$. Also, if $n\geq2$, then  the 3-Sylow subgroup $H_3$ of $H$ exists, and is abelian; otherwise, for all $x \in H_3 \setminus Z(H_3)$, the induced subgraph $\Gamma_c (G)[xZ(H_3)] \cong K_{|Z(H_3)|}$ would contain a 3-cycle (noting that $|Z(H_3)| \geq 3$). This, however, implies that  the induced subgraph $\Gamma_c (G)[H_3 \setminus \{1\}] \cong K_{3^n -1}$ contains a 3-cycle.  Therefore, we have $n\leq 1$. Thus, every finitely generated subgroup of $G$ is of order at most 6.  It follows that $G$ itself is of order not exceeding 6. Since $G$ non-abelian, we have  $G \cong S_3$.

\noindent \textit{Case 2.} \ $|Z(G)| = 2$.

In this case, $G/Z(G)$ is an elementary abelian 2-group; otherwise, for all $x\in G \setminus Z(G)$ with $x^2 \notin Z(G)$, the induced subgraph $\Gamma_c (G)[xZ(G)\sqcup x^2 Z(G)] \cong K_{2|Z(G)|}$ would contain a 3-cycle.  It follows that every element of $G$ is of order 2 or 4.  Since $G$ is non-abelian, there is an element $x \in G$ of order 4, and so, we have $Z(G) = \{1, x^2 \}$. It is easy to see that $C_G (x) = \langle x \rangle$; otherwise $\{x, x^3, w \}$ would form a 3-cycle in $\Gamma_c (G)$ for all $w \in C_G (x) \setminus \langle x \rangle$. Thus, $|C_G (x)|  = 4$. Let $z \in Cl_G (x) \setminus \{ x \}$.  Then,  $1 \neq zx^{-1} \in G' \subseteq Z(G)$. Therefore,  we have $zx^{-1}=x^2$, and so, $z = x^3$.  Thus, $Cl_G (x) = \{x, x^3 \}$, and so, $|G:C_G (x)|= |Cl_G (x)|=2$.  It follows that $|G|=8$. Since $G$ is non-abelian, we have $G \cong Q_8 \text{ or } D_8$  . This completes the proof.
\end{proof}
It follows, in particular,  from the above result that the girth of the commuting graph of a non-abelian group is 3 or $\infty$. Our next result is used not only in this section but also in the forthcoming sections.

\begin{prop}\label{bound} Let $G$ be a finite non-abelian group whose commuting graph has genus $g$, where $g$ is a non-negative integer.  Then the following assertions hold:  
\begin{enumerate}
\item   If $\emptyset \neq S \subseteq G \setminus Z(G)$ such that $xy=yx$ for all $x,y \in S$, then $|S|\leq \lfloor\frac{7+\sqrt{1+48g}}{2}\rfloor$. 
\item   $|Z(G)|\leq \frac{1}{t-1}\lfloor\frac{7+\sqrt{1+48g}}{2}\rfloor$, where $t = \max \{o(xZ(G))\mid xZ(G) \in G/Z(G)\}$.  
\item   If $A$ is an abelian subgroup of $G$, then $|A| \leq \lfloor\frac{7+\sqrt{1+48g}}{2}\rfloor +|A \cap Z(G)|$.
\end{enumerate}
 \end{prop}
\begin{proof}
Consider the induced subgraph $\Gamma_c (G)[S] \cong K_{|S|}$. If $g=0$, then $\gamma(K_{|S|}) = \gamma(\Gamma_c (G)[S])\leq \gamma(\Gamma_c (G))=0$, and so, it follows that $|S|\leq 4$. On the other hand, if $g>0$, then, by Heawood's formula \cite[Theorem 6.3.25]{west}, we have $|S| = \omega(\Gamma_c (G)[S]) \leq \omega(\Gamma_c (G)) \leq \chi(\Gamma_c (G)) \leq \lfloor \frac{7+\sqrt{1+48g}}{2}  \rfloor$. This proves (a). The remaining two assertions follow from (a); in fact, for (b) we take  $S= \overset{t-1}{\underset{i=1}\bigsqcup}   y^i Z(G)$, where $y \in G \setminus Z(G)$ such that $o(yZ(G))=t$, and for (c) we simply note that $A=(A\setminus Z(G))\cup(A\cap Z(G))$.
\end{proof}

Our third result of this section says that every collection of finite non-abelian groups whose commuting graphs have the same genus is finite.
\begin{thm}\label{fingen}
The order of a finite non-abelian group is bounded by a function of the genus of its commuting graph. Consequently, given a non-negative integer $g$,  there are at the most finitely many finite non-abelian groups whose commuting graphs  have genus $g$.
\end{thm}
\begin{proof} Let $G$ be a finite non-abelian group whose commuting graph  has  genus $g$. Let us put $h=\lfloor\frac{7+\sqrt{1+48g}}{2}\rfloor$.   Then, by Proposition \ref{bound}(a),  we have $|Z(G)| \leq h$. Let $p$ be a prime divisor of $|G|$, and $P$ be a Sylow $p$-subgroup of $G$ with $|P|=p^n$, where $n$ is a positive integer.  If $P \subset Z(G)$, then $|P| \leq h$.  So, let $P \nsubseteq Z(G)$. If $P$ is abelian, then, by Proposition \ref{bound}, we have $|P \setminus Z(G)| \leq h$, and hence, $|P| \leq 2h$. So, we assume that $P$ is non-abelian.  Then, $|Z(P)|=p^c$ for some positive integer $c<n$, and, by   \cite[Section I, Para 4]{burn}, $P$ has an abelian subgroup $A$ of order $p^v$,  where $v$ is a positive integer such that $v \geq -\frac{1}{2} + \sqrt{2n+c^2-c+\frac{1}{4}}$; in particular, $n<(2v+1)^2$.  By Proposition \ref{bound}, we have $p^v=|A| \leq 2h$; in particular, $v< 2h$ and $p < 2h$. Hence, it follows that $|P| = p^n < (2h)^{(4h+1)^2}$. Since the number of primes less than $2h$ is at most $h$, we have $|G|< (2h)^{h(4h+1)^2}$. This completes the proof.
\end{proof}

Recall that a group is said to be an {\em AC-group} if the centralizer of each of its non-central elements is abelian. The $AC$-groups have been extensively studied by many authors (see, for example, \cite{sch}, \cite{roc}, \cite{aam}).  Our final result of this section deals with finite non-abelian $AC$-groups.

\begin{prop}\label{prop-ac}
 Let $G$ be a finite non-abelian $AC$-group. Then 
 \[ {\textstyle
 \gamma(\Gamma_c (G)) = \underset{X \in \mathcal{P}} \sum \gamma(K_{|X|})},
 \]
where $\mathcal{P}=\{C_G(u)\setminus Z(G) \mid u\in G\setminus Z(G)\}$.  \end{prop}

\begin{proof}
Let $X \in \mathcal{P}$. Then, $X = C_G(u) \setminus Z(G)$  for some $u \in G$. If $x,y \in X$  such that $x\neq y$, then $[x,u]=[y,u]=1$, and so, by \cite[Lemma 3.2]{roc}, we have $[x,y]=1$.  Also, if $x \in X$ and $y \in G\setminus Z(G)$ such that $[x,y]=1$, then,  by \cite[Lemma 3.2]{roc}, we have $[y,u]=1$, and so, $y \in X$.  It follows that the  induced subgraph $\Gamma_c (G)[X]\cong K_{|X|}$ is a block of $\Gamma_c (G)$, and, since $G\setminus Z(G) = \underset{X \in \mathcal{P}}{\cup}X$, the collection $\{\Gamma_c (G)[X] \mid X\in \mathcal{P}\}$ is the block decomposition of $\Gamma_c (G)$. Therefore, by   \cite[Corollary 1]{bhky}, we have  $\gamma(\Gamma_c (G)) =  {\underset{X \in \mathcal{P}}\sum} \gamma(K_{|X|})$.
\end{proof}

\begin{rem}\label{rem-ac}
If $G$ is a finite non-abelian $AC$-group and $A$ is a finite abelian group, then $A \times G$ is also a finite non-abelian $AC$-group with  $C_{A \times G} (a,u) \setminus Z(A \times G) =$ $A \times (C_G(u)\setminus Z(G))$ for all $(a,u) \in (A \times G) \setminus Z(A \times G)$.  Therefore, it follows from Proposition \ref{prop-ac} that 
\[ \textstyle{
\gamma(\Gamma_c (A \times G)) = \underset{X \in \mathcal{P}}\sum \gamma(K_{|A||X|})},
\]
where $\mathcal{P}=\{C_G(u)\setminus Z(G) \mid u\in G\setminus Z(G)\}$.
\end{rem}

\section{Genus of the commuting graphs of some well-known $AC$-groups} 

In this section, we determine the genus of the commuting graphs of some well-known finite non-abelian $AC$-groups. Some of the results obtained here play crucial role in the study of planarity and toroidality of the commuting graphs of finite non-abelian groups.

\begin{prop} \label{dihed} The genus of the commuting graph of the dihedral group $D_{2n}=\langle x,y \mid y^n=x^2=1, xyx^{-1}=y^{-1}\rangle$, where  $n \geq 3$, is given by
\[ \gamma(\Gamma_c (D_{2n})) = \begin{cases} 
      \gamma(K_{n-2}) & \textrm{ if $n$ is even,} \\
      \gamma(K_{n-1}) & \textrm{ if $n$ is odd.} \\
   \end{cases} \]
\end{prop}
\begin{proof}
Note that $D_{2n}$ is a non-abelian $AC$-group. If $n$ is even, then  $Z(D_{2n}) = \{1,y^\frac{n}{2}\}$, $C_{D_{2n}}(y^i) = \langle y \rangle$ for $1\leq i \leq n-1$  ($i \neq \frac{n}{2}$),  and $C_{D_{2n}}(xy^j)=\{1,xy^j,y^\frac{n}{2},xy^{j+\frac{n}{2}}\}$  for $0\leq j \leq n-1$. If $n$ is odd, then  $Z(D_{2n}) = \{1\}$, $C_{D_{2n}}(y^i) = \langle y \rangle$  for $1\leq i \leq n-1$, and $C_{D_{2n}}(xy^j)=\{1,xy^j\}$ for $0\leq j \leq n-1$. Thus, if $n$ is even, the distinct centralizers of the non-central elements in $D_{2n}$ are $\langle y \rangle$ and $\{1,xy^j,y^\frac{n}{2},xy^{j+\frac{n}{2}}\}$,  where $0\leq j \leq \frac{n}{2}-1$, and so, by Proposition \ref{prop-ac}, we have $\gamma(\Gamma_c (D_{2n})) = \gamma(K_{n-2}) + \frac{n}{2}\gamma(K_2)= \gamma(K_{n-2})$.  On the other hand, if $n$ is odd, the distinct centralizers in $D_{2n}$ are $\langle y \rangle$ and $\{1,xy^j \}$, where  $0\leq j \leq n-1$, and so, by Proposition \ref{prop-ac}, we have $\gamma(\Gamma_c (D_{2n})) = \gamma(K_{n-1}) + n\gamma(K_1)= \gamma(K_{n-1})$. 
\end{proof}

\begin{prop} \label{quat} The genus of the commuting graph of the dicyclic group or the generalized quaternion group  $Q_{4n} = \langle x,y \mid y^{2n} = 1, x^2=y^n,xyx^{-1} = y^{-1}\rangle$, where  $n \geq 2$,  is given by  
\[\gamma(\Gamma_c (Q_{4n})) = \gamma(K_{2(n-1)}).\]	
\end{prop}
\begin{proof}
It is well-known that $Q_{4n}$ is a non-abelian $AC$-group with $Z(Q_{4n}) = \{1,y^n\}$, $C_{Q_{4n}}(y^i) = \langle y \rangle$ for  $1\leq i \leq 2n-1$ ($i \neq n$), and $C_{Q_{4n}}(xy^j) = \{1,xy^j,y^n,xy^{j+n}\}$  for  $0 \leq j \leq 2n-1$.  Therefore, the distinct centralizers of the non-central elements in $Q_{4n}$ are $\langle y \rangle$ and $\{1,xy^j,y^n,xy^{j+n}\}$, where  $0\leq j \leq n-1$, and so, by Proposition \ref{prop-ac}, we have $\gamma(\Gamma_c (Q_{4n})) = \gamma(K_{2(n-1)}) + n\gamma(K_2)= \gamma(K_{2(n-1)})$.  
\end{proof}

\begin{prop} \label{sdih}  
The genus of the commuting graph of the semidihedral group $SD_{2^n}$ $= \langle r,s \mid r^{2^{n-1}} = s^2 = 1,srs = r^{2^{n-2}-1}\rangle$, where $n \geq 4$, is given by 
\[\gamma(\Gamma_c (SD_{2^n})) = \gamma(K_{2^{n-1}-2}).\]
\end{prop}
\begin{proof}
$SD_{2^n}$ is a non-abelian $AC$-group with $Z(SD_{2^n}) = \{ 1, r^{2^{n-2}}\}$,  $C_{SD_{2^n}}(r^i) = \langle r \rangle$, for $1\leq i \leq 2^{n-1}-1$ ($i \neq 2^{n-2}$), and   $C_{SD_{2^n}}(sr^j) = \{1,sr^j,r^{2^{n-2}},sr^{j+2^{n-2}}\}$  for $0 \leq j \leq 2^{n-1}-1$.  Therefore, the distinct centralizers of the non-central elements in $SD_{2^n}$ are $\langle r \rangle$ and 
$\{1,sr^j,r^{2^{n-2}},sr^{j+2^{n-2}}\}$, where $0 \leq j \leq 2^{n-2}-1$,
and so, by Proposition \ref{prop-ac}, we have $\gamma(\Gamma_c (SD_{2^n})) = \gamma(K_{2^{n-1}-2}) + 2^{n-2}\gamma(K_2)= \gamma(K_{2^{n-1}-2})$.  
\end{proof}

\begin{prop}\label{nonpq}  The genus of the commuting graph of a non-abelian group $G$ of order $pq$, where $p$ and $q$ are primes with  $p \mid q-1$, is given by
\[\gamma(\Gamma_c (G)) = \gamma(K_{q-1}) + q \gamma(K_{p-1}).\]
\end{prop}

\begin{proof}
Note that $G$ is an $AC$-group with $|Z(G)|= 1$, in which the centralizers of the non-central elements are precisely the Sylow subgroups of $G$, and so, the result follows from Proposition \ref{prop-ac}. 
\end{proof}

\begin{prop}\label{nonp3}  The genus of the commuting graph of a non-abelian group $G$ of order $p^3$, where $p$ is a prime, is given by
\[\gamma(\Gamma_c (G)) = (p+1)\gamma(K_{p(p-1)}).\]
\end{prop}

\begin{proof}
Note that $G$ is an $AC$-group with $|Z(G)|= p$, in which the centralizers of the non-central elements are of order $p^2$. Since any two distinct centralizers of the non-central elements of $G$ intersect at $Z(G)$, it follows that the number of such centralizers is $p+1$. Hence, the result follows from Proposition  \ref{prop-ac}. 
\end{proof}
 
\begin{prop}\label{psl2q}
The genus of the commuting graph of the projective special linear group 
 $PSL(2, 2^k)$, where  $k \geq 2$, is given by 
\[ \gamma(\Gamma_c (PSL(2, 2^k))) =  (2^k +1) \gamma(K_{2^k -1}) + 2^{k-1}(2^k +1) \gamma(K_{2^k -2}) +   2^{k-1}(2^k -1)\gamma(K_{2^k}).\]

\end{prop}
\begin{proof}
It is well-known that $PSL(2,2^k)$ is a non-abelian group of order $2^k(2^{2k} -1)$ with $Z(PSL(2,2^k))=\{1\}$. Moreover, in view of \cite[Proposition 3.21]{aam}, the following assertions hold for $PSL(2,2^k)$:
\begin{enumerate}
\item  $PSL(2,2^k)$ has an elementary abelian $2$-subgroup $P$ of order   $2^k$ such that the number of conjugates of $P$ in $PSL(2,2^k)$ is $2^k +1$.
\item  $PSL(2,2^k)$ has a cyclic subgroup $A$ of order $2^k -1$ such that the number of conjugates of $A$ in $PSL(2,2^k)$ is $2^{k-1}(2^k +1)$.
\item  $PSL(2,2^k)$ has a cyclic subgroup $B$ of order $2^k +1$ such that the number of conjugates of $B$ in $PSL(2,2^k)$ is $2^{k-1}(2^k -1)$.
\item The centralizers of the non-trivial elements of $PSL(2,2^k)$  constitute  precisely the family $\{xPx^{-1},xAx^{-1},xBx^{-1} \mid x \in G\}$; in particular, $PSL(2,2^k)$ is an $AC$-group. 
\end{enumerate}
Hence, the result follows from Proposition \ref{prop-ac}.
\end{proof}

\begin{prop}\label{gl2q}
 The genus of the commuting graph of the general linear group  $GL(2,q)$, where $q=p^n >2$ ($p$ is a prime), is given by 
\[ \gamma(\Gamma_c (GL(2,q)))=\frac{q(q+1)}{2}\gamma(K_{(q-1)(q-2)})+\frac{q(q-1)}{2}\gamma(K_{q(q-1)})+(q+1)\gamma(K_{(q-1)^2}).\]
\end{prop}

\begin{proof}
Note that   $GL(2,q)$  is a non-abelian $AC$-group (see \cite[Lemma 3.5]{aam}) with      $|GL(2,q)|= (q^2 -1)(q^2 -q)$ and $|Z(GL(2,q))|=q-1$.  Also, in view of \cite[Proposition 3.26]{aam}, the centralizers of the non-central elements of $GL(2,q)$ are precisely the members of the family $\{xDx^{-1},xIx^{-1},xPZ(GL(2,q))x^{-1} \mid x \in G\}$, where 
\begin{enumerate}
\item $D$ is the subgroup of $GL(2,q)$ consisting of all diagonal matrices, $|D|=(q-1)^2$,  and the number of conjugates of $D$ in $GL(2,q)$ is $\frac{q(q+1)}{2}$, 
\item  $I$ is a cyclic subgroup of $GL(2,q)$, $|I|= q^2 -1$, and the number of conjugates of $I$ in $GL(2,q)$ is $\frac{q(q-1)}{2}$, 
\item  $P$ is the Sylow $p$-subgroup of $GL(2,q)$ consisting of all upper triangular matrices with 1 in the diagonal,  $|PZ(GL(2,q))|=q(q-1)$, and the number of conjugates of $PZ(GL(2,q))$ in $GL(2,q)$ is $q+1$.
\end{enumerate}
Hence, the result follows from Proposition \ref{prop-ac}.
\end{proof}

In view Remark \ref{rem-ac} and the results obtained in this section, one can easily compute the genus of the commuting graph of the group $A \times G$, where $A$ is a finite abelian group and $G$ is any one of the groups considered in Propositions \ref{dihed} to \ref{gl2q}.

\section{Finite non-abelian groups whose commuting graphs are planar} 

In this section, we characterize all finite non-abelian groups whose commuting graphs are planar.  However, we begin the section with a lemma containing a couple of elementary properties of finite $2$-groups. 

\begin{lem}\label{2gzg4} Let $G$ be a finite 2-group. Then, the following assertions hold:
\begin{enumerate}
\item  If $|G|\geq 16$, then $G$ contain an abelian subgroup of order 8.
\item If $|G|\geq 32$ and $|Z(G)|\geq 4$, then $G$ contain an abelian subgroup of order 16.
\end{enumerate} 
\end{lem}
\begin{proof}
If $|G|=32$ and $|Z(G)|=4$, then, using GAP \cite{gap} or otherwise (see, for example \cite[Theorem 35.4]{berk}), it is not difficult to see that   $G$  contains an abelian subgroup of order 16.  The rest of the lemma follows immediately from \cite[Section I, Para 4]{burn}.  
\end{proof}

If $G$ is a finite non-abelian group whose commuting graph is planar, then,  by Proposition \ref{bound}(b), we have   $1 \leq |Z(G)| \leq 4$.  Our first result of this section provides some useful information regarding the size of $G$ and its abelian subgroups.
\begin{prop}\label{abel}
Let $G$ be a finite non-abelian group whose commuting graph is planar. Then the following assertions hold:
\begin{enumerate}
\item  If $p$ is a prime divisor of $|G|$, then $p\leq 5$. 
\item  Neither $9$ nor $25$ divides $|G|$, and hence, $|G|$ is even with $|G|\geq 6$.
\end{enumerate}
\end{prop}

\begin{proof}
If $p \geq 7$ is a prime divisor of $|G|$, then $G/Z(G)$ has an element of order $p$, and so, by Proposition \ref{bound}(b), we have $|Z(G)|\leq \frac{4}{p-1}<1$, which is impossible. This proves (a).  For (b), note that if $9$ or $25$ divides $|G|$, then, a Sylow $3$-subgroup or a Sylow $5$-subgroup of $G$ contains a  subgroup of order $9$ or $25$. Since such a subgroup is abelian, we have, in view of Proposition  \ref{bound}(c), a contradiction in either situation.   That $|G|$ is even with $|G|\geq 6$, follows from the fact that $G$ is non-abelian.
\end{proof}

Given a finite non-abelian group $G$, whose commuting graph is planar,  it follows from Proposition \ref{abel} that  $|G|=2^r 3^s 5^t$, where $r \geq 1$ and $s,t \in \{0, 1\}$.  However, depending on the values of $|Z(G)|$, the range of possible values of $|G|$ gets reduced further.

\begin{prop}\label{plan}
Let $G$ be a finite non-abelian group whose commuting graph is planar. Then the possible values of $|G|$ are given as follows:
\begin{enumerate}
\item  If $|Z(G)| =1$, then $|G|=2^r 3^s 5^t$, where $1 \leq r \leq 3$ and $s,t \in \{0, 1\}$. 
\item  If $|Z(G)| =2$, then $|G| \in \{8, 12, 24\}$. 
\item  If $|Z(G)| =4$, then $|G| =16$.   
\item  $|Z(G)| \neq 3$.
\end{enumerate}
\end{prop}

\begin{proof}
We have $|G|=2^r 3^s 5^t$, where $r \geq 1$ and $s,t \in \{0, 1\}$. Let $H$ be a Sylow $2$-subgroup of $G$.  If $|Z(G)| \leq 3$ and  $r \geq 4$, then, by Lemma \ref{2gzg4}(a), $H$  has an abelian subgroup of order $8$. However,  by Proposition \ref{bound}(c), the size of an abelian subgroup of $G$ does not exceed $7$ if $|Z(G)| \leq 3$.  Thus, $r \leq 3$ if $|Z(G)| \leq 3$. On the other hand, if $|Z(G)| =4$ and $r \geq 5$, then, using Lemma \ref{2gzg4}(b) and  noting that  $Z(G) \subseteq Z(H)$, there is an abelian subgroup of $H$ of order $16$. But, by Proposition \ref{bound}(c), this is impossible.  Thus, $r \leq 4$ if $|Z(G)| =4$.  If $5$ divides $|G|$, then $G/Z(G)$ has an element of order $5$, and so, by Proposition \ref{bound}(b), we have $|Z(G)| =1$.  Also, if $|Z(G)| =4$, then $3$ does not divide $|G|$; otherwise $G/Z(G)$ would have an element of order $3$, which, by Proposition \ref{bound}(b), is impossible.  Now, it is a routine matter to see that  the assertions (a), (b) and (c)  hold. Finally, note that if $|Z(G)| =3$, then, by the above argument, we have $|G|=12$ or $24$. Therefore, $G$ has a subgroup $A$ of order $4$, and hence,  an abelian subgroup $AZ(G)$ of order 12, which, by Proposition \ref{bound}(c), is impossible.  Thus, (d) holds as well.
\end{proof}

Note that some of the possibilities mentioned in Proposition \ref{plan} are not  maintainable; for example, in (a), it is obviously not possible to have   $s=t= 0$. In fact, the following small result helps us in avoiding few more finite groups as far as the planarity of their commuting graphs is concerned. 

\begin{prop}\label{rid}
Let $G$ be a finite non-abelian group. If $|G|= 30$, or if $G$ is a solvable group with $|G| = 60$ or $120$, then  $G$ has an subgroup of order $15$ (which is obviously abelian).  Also, if $|G|= 40$, then $G$ has an abelian subgroup of order $10$. 
\end{prop}

\begin{proof}
If $|G| = 30$, or if $G$ is a solvable group with $|G| = 60$ or $120$, then, by a theorem of Hall (see \cite[Theorem 5.28]{rot}), $G$ has a subgroup of order $15$. On the other hand, if $|G| = 40$, then $G$ has a unique Sylow $5$-subgroup, and so, considering the centralizer and the number of conjugates of an element of order $5$, one can show that $G$ has an element (hence, an abelian subgroup) of order $10$.    
\end{proof}

In view of Proposition \ref{bound}(c) and Proposition \ref{plan}, it follows from proposition \ref{rid} that if $G$ is a finite non-abelian group whose commuting graph is planar, then $|G| \notin \{30, 40\}$; in addition, if $G$ is solvable, then  $|G| \notin \{60, 120\}$.

We also have the following useful result concerning the groups of order $16$.
\begin{prop}\label{g16}
Let $G$ be a finite non-abelian group with  $|Z(G)|=4$.  Then, the commuting graph of $G$ is planar if and only if $|G|=16$.
\end{prop}
\begin{proof}  
Let $G$ be a finite group with $|G|=16$ and $|Z(G)|=4$.  Note that, for each $x \in G \setminus Z(G)$,  we have $|C_G(x)|=8$ and $C_G(x)= \langle x \rangle Z(G)$, which is abelian. Thus, $G$ is an $AC$-group with $|C_G(x) \setminus Z(G)|=4$. Hence, it follows from Proposition \ref{prop-ac} that  $\gamma(\Gamma_c (G)) = 0$, that is, the  commuting graph of $G$ is planar.  This, in view of Proposition \ref{plan}(c),  completes the proof.   
\end{proof}

\begin{rem}\label{rem16}
Up to isomorphism, there are exactly six non-abelian groups of order $16$ with centers of order  $4$, namely, the two direct products $\mathbb{Z}_2 \times D_8$ and $\mathbb{Z}_2 \times Q_8$, the Small Group 
$SG(16,3)= \langle a,b \mid a^4 = b^4 =1, ab=b^{-1}a^{-1}, ab^{-1}=ba^{-1} \rangle$,  
 the semi-direct product $\mathbb{Z}_4 \rtimes \mathbb{Z}_4 = \langle a,b \mid a^4 = b^4  = 1, bab^{-1}=a^{-1} \rangle$,  
  the central product $D_8 * \mathbb{Z}_4 =  \langle a,b,c \mid a^4=b^2=c^2=1, ab=ba, ac=ca, bc=a^2 cb \rangle$ 
  and the modular group $M_{16} =  \langle a,b \mid a^8 = b^2  = 1, bab=a^5 \rangle  $.
\end{rem}

We now state and prove the main result of this section, where two new groups make their appearance, namely, 
the Suzuki group $Sz(2) = \langle a,b \mid a^5 = b^4  = 1, bab^{-1}=a^{2} \rangle$,  
and the special linear group $SL(2,3)=  \langle a,b,c \mid a^3 = b^3 = c^2 = abc \rangle $.

\begin{thm}\label{mainplan}
Let $G$ be a finite non-abelian group. Then, the commuting graph of $G$ is planar if and only if $G$ is isomorphic to either $ S_3$, $D_{10}$, $A_4$, $Sz(2)$, $S_4$,  $A_5$, $D_8$, $Q_8$,  $D_{12}$, $Q_{12}$,   $SL(2,3)$,  $\mathbb{Z}_2 \times D_8$, $\mathbb{Z}_2 \times Q_8$, $SG(16,3)$, $\mathbb{Z}_4 \rtimes \mathbb{Z}_4$, $D_8 * \mathbb{Z}_4$  or $M_{16}$.
\end{thm}

\begin{proof}
In view of Proposition \ref{plan}, Proposition \ref{g16}, Remark \ref{rem16} and the para following Proposition \ref{rid}, it is enough to study the planarity of the commuting graph of a finite group $G$ that belongs to one of the following categories:
\begin{enumerate}
\item[I.] $|Z(G)|=1$ and $|G| \in \{6,10,12,20,24\}$.
\item[II.] $|Z(G)|=1$, $|G| \in \{60,120\}$ and $G$ is not solvable.
\item[III.] $|Z(G)|=2$ and $|G| \in \{8,12,24\}$.
\end{enumerate}

We use GAP \cite{gap} to examine the groups that belong to the above categories and look into  some of their properties which eventually help in  concluding whether their commuting graphs are planar or not. 

There are exactly five groups that belong to category I, namely, $S_3$, $D_{10}$, $A_4$, $Sz(2)$ and $S_4$. Among these groups, $S_3$, $D_{10}$, $A_4$ and $Sz(2)$ are $AC$-groups such that, in each case, the size of the  centralizer of every non-central element is at most $5$, and so, by Proposition \ref{prop-ac}, the commuting graph of each of these groups is planar; on the other hand, the commuting graph $\Gamma_c (S_4)$  has a block decomposition given by  
\[ \Gamma_c (S_4)[H] \cup \underset{\sigma \in \mathcal{F}} \cup \Gamma_c (S_4)[H_{\sigma}],
\]
where $\mathcal{F}= \{(1\, 2),(1\, 3),(1\, 4),(1\, 2\, 3\, 4),(1\, 2\, 4\, 3),(1\, 3\, 2\, 4),(1\, 2\, 3),(1\, 2\, 4),(1\, 3\, 4),(2\, 3\, 4)\}$,  $H= \{(1\, 2)(3\, 4),(1\, 3)(2\, 4),(1\, 4)(2\, 3)\}$ and $H_{\sigma} = C_{\sigma} (S_4) \setminus \{(1)\}$  for all $\sigma \in \mathcal{F}$,   and so, by \cite[Corollary 1]{bhky}, it follows that $\gamma(\Gamma_c (S_4))= 7\gamma(K_3) + 4\gamma(K_2) = 0$.

There are  exactly two groups that belong to category II, namely, $A_5$ and $S_5$.  Of the two groups, $A_5$ is an $AC$-group in which the  centralizer of every non-central element  is at most $5$, and so, by Proposition \ref{prop-ac},  its commuting graph is planar; on the other hand, $S_5$ has an abelian subgroup of order $6$, namely, $C_{S_5} (1\,2) = \langle (1\,2), (3\,4\,5) \rangle$, and so, by Proposition \ref{bound}(c), its commuting graph is not planar.

Finally, there are exactly nine groups that belong to category III. However, except $D_8$, $Q_8$, $D_{12}$, $Q_{12}$  and $SL(2,3)$, each of the remaining four groups has an abelian centralizer of order at least $8$, and so, by Proposition \ref{bound}(c), has commuting graph of positive genus. The groups $D_8$, $Q_8$, $D_{12}$, $Q_{12}$  and $SL(2,3)$, on the other hand, are all $AC$-groups such that, in each case, the size of the  centralizer of every non-central element is at most  $6$, and so, by Proposition \ref{prop-ac}, the commuting graph of each of these groups is planar.  This completes the proof.
\end{proof}

\section{Finite non-abelian groups whose commuting graphs are toroidal} 

In this section, we characterize all finite non-abelian groups whose commuting graphs are toroidal.  

 The following result is analogous to Proposition \ref{abel}.

\begin{prop}\label{torzg}
Let $G$ be a finite non-abelain group whose commuting graph is toroidal. Then, the following assertions hold:
\begin{enumerate}
\item  $|Z(G)|\leq 3$.
\item  If $p$ is a prime divisor of $|G|$, then $p\leq 7$. 
\item  None of $25$, $27$ and $49$ is a divisor of $|G|$.
\end{enumerate} 
\end{prop}

\begin{proof}
Suppose that $|Z(G)| = 4$. If $p$ is an odd prime divisor of $|G|$, then $G/Z(G)$ has an element of order at least $3$, and so, by Proposition \ref{bound}(b), we have a contradiction. Therefore, in view of Proposition \ref{g16}, $|G|=2^r$ for some $r \geq 5$. But, by Lemma \ref{2gzg4}(b) and Proposition \ref{bound}(c),  we again have a contradiction. So, let   $|Z(G)|\geq 5$. Choose $x,y \in G \setminus Z(G)$  such that $xy \neq yx$. Then, $xZ(G)$ and $yZ(G)$ are two disjoint subsets of $G \setminus Z(G)$, and the induced subgraph $\Gamma_c (G)[xZ(G)] \cong K_m \cong \Gamma_c (G)[yZ(G)]$, where $m=|Z(G)|$. Hence, by Lemma \ref{lem1} and Lemma  \ref{kn}, it  follows that $\gamma(\Gamma_c (G))\geq 2$, which is impossible.   Thus, (a) holds.  

 If $p \geq 11$ is a prime divisor of $|G|$, then, by (a), there is an element of order $p$ in $G/Z(G)$. Therefore, by Proposition \ref{bound}(b), we have $|Z(G)|\leq \frac{7}{p-1}<1$, which is impossible. This proves (b). 
 
For (c), note that if $25$ or $49$ divides $|G|$, then $G$ has an abelian subgroup of order $25$ or $49$. Since such a subgroup is obviously abelian, we have a contradiction according to (a) and Proposition  \ref{bound}(c).  On the other hand, if $27$ divides $|G|$, then $G$ has a subgroup of order $27$. Therefore, since the commuting graph of a subgroup of $G$ is a subgraph of the commuting graph of $G$, we have, by Proposition \ref{nonp3},  a contradiction. This completes the proof.
\end{proof}

Analogous to Proposition \ref{g16}, we also have the following result concerning the groups of order $16$.

\begin{prop}\label{2gpz2}
Let $G$ be a finite non-abelian $2$-group with $|Z(G)|=2$.  Then, the commuting graph of $G$ is toroidal if and only if $|G|=16$, that is, if and only if $G$ is isomorphic to either $D_{16}$, $Q_{16}$ or $SD_{16}$.
\end{prop}

\begin{proof}
Let $|G|\geq 32$. Then, by the class equation \cite[page 74]{rot},  there exists $x \in G \setminus Z(G)$ such that $|G:C_G (x)|=2$, and so, $|C_G (x)| \geq 16$. Clearly $|Z(C_G (x))| \geq 4$.  First,  let us assume that $|Z(C_G (x))| = 4$. Let $v \in C_G (x) \setminus Z(C_G (x))$. Then, there exists $w \in C_G (x) \setminus Z(C_G (x))$ such that $vw \neq wv$. Let $z$ denote the non-trivial element of $Z(G)$.  Consider the two disjoint subsets $H_1= \{x,v,vz,xv,xvz\}$  and $H_2= \{xz,w,wz,xw,xwz\}$ of $G\setminus Z(G)$. Clearly, the induced subgraph $\Gamma_c (G)[H_1] \cong K_5 \cong \Gamma_c (G)[H_2]$.  Hence, by Lemma \ref{lem1} and Lemma  \ref{kn}, it  follows that $\gamma(\Gamma_c (G))\geq 2$.  Next, let us assume that    $|Z(C_G (x))| \geq 8$. Consider a subset $V$  of  $Z(C_G (x)) \setminus Z(G)$ such that $|V|=3$ and put $W = C_G (x) \setminus (V \cup Z(G))$.  Clearly,  the induced subgraph $\Gamma_c (G)[V\cup W]$ has a subgraph isomorphic to the complete bipartite graph $K_{3,n}$, where $n = |C_G (x)| - 5 \geq 11$.  This, by Lemma \ref{knm}, implies that the genus of the commuting graph of $G$ is at least $3$.  Thus, in view of Theorem \ref{mainplan}, it follows that if the commuting graph of $G$ is toroidal, then $|G|=16$.    On the other hand, it is well-known (using GAP\cite{gap}, for example) that if   $|G|=16$ and $|Z(G)|=2$, then $G$ is isomorphic to either $D_{16}$, $Q_{16}$ or $SD_{16}$, and, by  Proposition \ref{dihed}, Proposition \ref{quat} and Proposition \ref{sdih},  the commuting graph of each of these groups is toroidal.  This completes the proof.
\end{proof}

We also have the following result concerning the finite groups that are not $2$-groups.

\begin{prop}\label{not2g}
Let $G$ be a finite non-abelian group with $|G|=2^r m$, where $r\geq 0$, $m>1$ and $m$ is odd.  If the commuting graph of $G$ is toroidal, then $r \leq 3$.
\end{prop}

\begin{proof}
Suppose that the commuting graph of $G$ is toroidal and that $r \geq 4$. Let $H$ be a sylow $2$-subgroup of $G$. In view of Proposition \ref{bound}(c), $H$ is non-abelian.  Moreover, the commuting graph of $H$, being a subgraph of the commuting graph of $G$, is either planar or toroidal. 

\noindent \textit{Case 1.} $\Gamma_c (H)$ is planar.  

In this case, by Proposition \ref{plan}, we have $|Z(H)|=4$. Therefore, by Proposition \ref{torzg}(a), we have $|Z(H) \setminus Z(G)|\geq 2$. Let $v_1,v_2 \in Z(H) \setminus Z(G)$ such that $v_1 \neq v_2$. Also, let $x,y \in H \setminus Z(H)$ such that $xy \neq yx$. Then, it is easy to see that   $\{v_1 \}\cup xZ(H)$ and $\{v_2 \}\cup yZ(H)$ are two disjoint subsets of  $G \setminus Z(G)$, and
 the induced subgraph $\Gamma_c (G)[\{v_1 \}\cup xZ(H)] \cong K_5 \cong \Gamma_c (G)[\{v_2\}\cup yZ(H)]$.  This implies  that $\gamma(\Gamma_c (G)) \geq 2$, which is a contradiction. 

\noindent \textit{Case 2.} $\Gamma_c (H)$ is toroidal.

In this case, by Proposition \ref{torzg}(a), we have $|Z(H)|=2$. Therefore, by Proposition \ref{2gpz2}, we have  $|H|=16$ and there exists an element $x \in H$ with $o(x)=8$. Note that,   for each $y \in G$ with $o(y)=8$,   we have $\langle x \rangle = \langle y \rangle$; otherwise, choosing $M=\{x,x^3,x^5,x^7, x^2 \}$ and $N=\{y,y^3,y^5,y^7, w\}$ with $w \in \{y^2, y^6\}\setminus \{x^2\}$, we would have the induced subgraph $\Gamma_c (G)[M]\cong K_5 \cong \Gamma_c (G) [N]$,   which, by Lemma \ref{lem1} and Lemma \ref{kn},  implies that $\gamma(\Gamma_c (G)) \geq 2$, a contradiction.  Also, in view of Proposition \ref{bound}, we have $|C_G (x)|=8$; otherwise either $Z(H)$ would have an element of order $8$  or   $G$ would have an abelian subgroup of order at least $24$.  Hence, it follows that  the number of conjugates of $x$ in $G$ is  $2m \geq 6$, that is, there are at least six elements of order $8$ in $G$.  This contradiction completes the proof.
\end{proof}

If $G$ is a finite non-abelian group whose commuting graph is toroidal, then it follows from Proposition \ref{torzg} that $|G|=2^r 3^s 5^t 7^u$, where $r \geq 0$, $0\leq s \leq 2$ and $t,u \in \{0, 1\}$. However, as in proposition \ref{plan}, the range of possible values of $|G|$ gets reduced further depending on the values of $|Z(G)|$.

\begin{prop}\label{toroi}
Let $G$ be a finite non-abelian group whose commuting graph is toroidal.  Then the possible values of $|G|$ are given as follows:
\begin{enumerate}
\item  If $|Z(G)| =1$, then  $|G|=2^r 3^s 5^t 7^u$ where $0 \leq r \leq 3$ and  $s,t,u \in \{0, 1\}$.
\item  If $|Z(G)| =2$, then  $|G| \in \{16, 24 \}$.
\item  If $|Z(G)| =3$, then   $|G|=18$.
\end{enumerate}
\end{prop}

\begin{proof}
If $5$ or $7$ divides $|G|$, then $G/Z(G)$ has an element of order $5$ or $7$, and so, by Proposition \ref{bound}(b), we have $|Z(G)| =1$.  If $|Z(G)| \leq 2$, then $9$ does not divide $|G|$; otherwise $G$ would have an abelian subgroup $T$ of order $9$, which, by Proposition \ref{bound}(c), is impossible noting that $|T\cap Z(G)|=1$.   If $|Z(G)|=3$, then $4$ does not divide $|G|$; otherwise $G$ would have a subgroup $A$ of order $4$, and hence, an abelian subgroup $AZ(G)$ of order 12, which, by Proposition \ref{bound}(c), is impossible.   In view of Theorem \ref{mainplan}, proposition \ref{2gpz2} and Proposition \ref{not2g}, it is now not difficult to see that all the three assertions hold.
\end{proof}

Needless to mention that some of the possibilities mentioned in Proposition \ref{toroi} are clearly not maintainable; for example, in (a), it is impossible to have  $s=t=u= 0$, $r=u=0$ or $r=s=0$. Moreover, in view of Proposition \ref{bound}(c) and Proposition \ref{toroi}, it follows from proposition \ref{rid} that if $G$ is a finite non-abelian group whose commuting graph is toroidal, then $|G| \notin \{30, 40\}$; in addition, if $G$ is solvable, then  $|G| \notin \{60, 120\}$.

 The following result, along with Proposition \ref{rid},  helps us in rejecting some more possibilities.

\begin{prop}\label{rid7}
Let $G$ be a finite non-abelian group whose commuting graph is toroidal. If $|G| = 7m$, where $m \geq 2$ and $7\nmid m$, then $m=2$ or $3$.
\end{prop}

\begin{proof}
By Proposition \ref{toroi}, we have $|Z(G)|=1$. Let $H$ be a Sylow $7$-subgroups of $G$. If $S$ is a Sylow $7$-subgroups of $G$ such that $S\neq H$, then it is easy to see that the induced subgraph $\Gamma_c (G)[S\setminus Z(G)]\cong K_6 \cong \Gamma_c (G) [H \setminus Z(G)]$, and so, we have a contradiction to the toroidality of $\Gamma_c (G)$. Therefore, $H$ is the unique (hence, normal) Sylow $7$-subgroup of $G$. Note that $C_G (H) =H$; otherwise $C_G (H)$ (hence, $G$) would have  an element (hence, an abelian subgroup) of order at least $14$, which, by Proposition \ref{bound}(c), is impossible. Therefore, by $N/C$ Lemma \cite[Theorem 7.1(i)]{rot}, $G/H$ is isomorphic to a subgroup of the cyclic group $\mathbb{Z}_6 \cong {\rm Aut}(H)$. Since $|G/H|=m$, it follows that  $m|6$ and  $G$ has an element $x$ of order $m$. If $m=6$, then the induced subgraph $\Gamma_c (G)[\langle x \rangle \setminus Z(G)]\cong K_5$, and so, we have a contradiction to the toroidality of $\Gamma_c (G)$ since $\Gamma_c (G) [H \setminus Z(G)] \cong K_6$. Hence, we have $m=2$ or $3$.
\end{proof}

We now state and prove the main result of this section.

\begin{thm}\label{maintor}
Let $G$ be a finite non-abelian group. Then, the commuting graph of $G$ is toroidal if and only if $G$ is isomorphic to either $D_{14}$, $\mathbb{Z}_7 \rtimes \mathbb{Z}_3$, $\mathbb{Z}_2 \times A_4$, $\mathbb{Z}_3 \times S_3$, $D_{16}$, $Q_{16}$ or $SD_{16}$.
\end{thm} 

\begin{proof}
In view of  Proposition \ref{2gpz2}, Proposition \ref{toroi} (and the para following it), Proposition \ref{rid7} and the proof of Theorem \ref{mainplan}, it is enough to study the toroidality of the commuting graph of a finite group $G$ that belongs to one of the following categories:
\begin{enumerate}
\item[I.] $|Z(G)|=1$ and $|G| \in \{14,21\}$.
\item[II.] $|Z(G)|=1$, $|G| =120$ and $G$ is not solvable.
\item[III.] $|Z(G)|=2$, $|G| =24$ and $G \not\cong SL(2,3)$.
\item[IV.] $|Z(G)|=3$ and $|G|= 18$.
\end{enumerate}

As in the proof of Theorem \ref{mainplan}, we use GAP \cite{gap} to  determine the groups belonging to the above categories whose commuting graphs are toroidal.

$D_{14}$ and $\mathbb{Z}_7 \rtimes \mathbb{Z}_3$ are the only groups that belong to category I and, by Proposition \ref{dihed} and Proposition \ref{nonpq}, the commuting graphs of these groups are toroidal.

$S_5$ is the only group that belongs to the category II. However, $S_5$  has two abelian subgroups  $S= \langle (1 \ 2)(3 \ 4 \ 5) \rangle$ and $T= \langle (4 \ 5)(1 \ 2 \ 3) \rangle$ such that  $|S|=|T|=6$ and  $S \cap T$ is trivial.  It follows that the commuting graph of $S_5$ is not toroidal.

There are  exactly four groups that belong to category III and all of them are $AC$-groups.  However, except $\mathbb{Z}_2 \times A_4$, each of the remaining three groups have an abelian centralizer of order $12$, whereas   $\mathbb{Z}_2 \times A_4$ has only one abelian centralizer of order $8$ and the rest of order $6$. Therefore, by Proposition \ref{prop-ac}, it follows that $\mathbb{Z}_2 \times A_4$ is the only group in category III whose commuting graph  is toroidal. 

 $\mathbb{Z}_3 \times S_3$ is the only group that belongs to the category IV and it is an $AC$-group with only one abelian centralizer of order $9$ and the rest of order $6$. Therefore, by Proposition \ref{prop-ac}, it follows that the commuting graph of $\mathbb{Z}_3 \times S_3$ is toroidal.  This completes the proof.
\end{proof}

\section*{Acknowledgements}
The first author is grateful to Group-Pub-Forum for many useful discussions. The second author  wishes to express his sincere thanks to  CSIR  (India) for its  financial assistance  (File No. 09/347(0209)/2012-EMR-I).


\begin{thebibliography}{33}

\bibitem{aam} A. Abdollahi, S. Akbari and H. R. Maimani, {\em Non-commuting graph of a group}, Journal of Algebra \textbf{298} (2006), 468--492.

\bibitem{akb1}
S. Akbari, M. Ghandehari, M. Hadian and A. Mohammadian, {\em On commuting graphs of semisimple rings}, Linear Algebra and its Applications  \textbf{390} (2004), 345--355.

\bibitem{akb2}
S. Akbari, A. Mohammadian, H. Radjavi and P. Raja, {\em On the diameters of commuting graphs}, Linear Algebra and its Applications  \textbf{418} (2006), 161--176.

\bibitem{bbhr} C. Bates, D. Bundy, S. Hart and P. Rowley, {\em A Note on Commuting Graphs for Symmetric Groups}, Electronic Journal of Combinatorics \textbf{16}  (2009), 1--13.

\bibitem{bhky}
J. Battle, F. Harary, Y. Kodama and  J. W. T. Youngs, {\em Additivity of the genus of a graph}, Bulletin of the American Mathematical Society \textbf{68}  (1962), 565--568.

\bibitem{berk}
Y. Berkovich, {\em Groups of Prime Power Order - Volume 1}, de Gruyter Expositions in Mathematics \textbf{46}, Walter de Gruyter, Berlin, New York  (2008).

\bibitem{bf}
R. Brauer and K. A. Fowler, {\em On groups of even order}, Annals of Mathematics \textbf{62}(3) (1955), 565--583.

\bibitem{burn}
W. Burnside, {\em On some properties of groups whose orders are powers of primes}, Proceedings of the London Mathematical Society \textbf{11}(2) (1912), 225-245.

\bibitem{das}
A. K. Das, H. R. Maimani, M. R. Pournaki, and S. Yassemi, {\em Nonplanarity of unit graphs and classification of the toroidal ones}, Pacific Journal of Mathematics (to appear).

\bibitem{fang}
M. Fang and P. Zhang, {\em Finite groups with graphs containing no triangles}, Journal of Algebra \textbf{264}   (2003) 613--619.

\bibitem{gp}
M. Giudici and C. Parker,  {\em There is no upper bound for the diameter of the commuting graph of a
finite group}, Journal of Combinatorial Theory (Series A)  \textbf{120}(7) (2013),  1600--1603.

\bibitem{mhall}
M. Hall, {\em The Theory of groups}, The Macmillan Company, New York, 1963.

\bibitem{ira1}
A. Iranmanesh  and A. Jafarzadeh  {\em Characterization of finite groups by their commuting graph}, Acta Mathematica Academiae Paedagogicae Ny\'{\i}regyh\'{a}ziensis  \textbf{23}(1)  (2007), 7--13.
 
\bibitem{ira2}
A. Iranmanesh  and A. Jafarzadeh  {\em On the commuting graph associated with the symmetric and alternating
groups}, Journal of Algebra and its Applications, \textbf{7}(1) (2008), 129--146.

\bibitem {mai} H. R. Maimani, C. Wickham, S. Yassemi, Rings whose total graphs have genus at most one, {\it Rocky Mountain Journal of Mathematics} {\bf 42}(5)  (2012),  1551--1560.

\bibitem{mog}
A. R. Moghaddamfar, W. J. Shi, W. Zhou  and A. R. Zokayi,  {\em On the non-commuting graph associated with a finite group}, Siberian Mathematical Journal  \textbf{46}(2)  (2005),  325--332.

\bibitem{mor}
G. L. Morgan and C. W. Parker, {\em The diameter of the commuting graph of a finite group with trivial center},  Journal of Algebra \textbf{393}(1) (2013), 41--59.

\bibitem{neu}
B.H. Neumann, {\em A problem of Paul Erd\"{o}s on groups}, Journal of the Australian Mathematical Society (Series A)   \textbf{21}  (1976), 467--472.

\bibitem{par}
C. Parker, {\em The commuting graph of a soluble group},  Bulletin of the London Mathematical Society, \textbf{45}(4)  (2013), 839--848.

\bibitem{rap}
A. S. Rapinchuk, Y.  Segev,   G. M.  Seitz,  {\em Finite quotient of the multiplicative group of a finite dimensional division algebra are solvable}, Journal of the American  Mathematical Society  \textbf{15} (2002), 929--978.

\bibitem{roc}
D.M. Rocke, {\em p-Groups with abelian centralizers}, Proceedings of the London Mathematical Society,   \textbf{30}(3)  (1975), 55--75.

\bibitem{rot}
J. J. Rotman, {\em An introduction to the theory of groups} (fourth edition), Springer-Verlag, New York, Inc., 1995.

\bibitem{ss}
Y. Segev and G. M. Seitz, {\em Anisotropic Groups of Type $A_n$ and the commuting graph of finite simple groups}, Pacific Journal of Mathematics 
\textbf{202} (2002), 125--225.

\bibitem{sch}
R. Schmidt, {\em Zentralisatorverb\"{a}nde endlicher gruppen}, Rendiconti del Seminario Matematico della Universit\`{a} di Padova  \textbf{44} (1970), 97--131.

\bibitem{wang}
H. J. Wang, {\em Zero-divisor graphs of genus one},  Journal of Algebra, \textbf{304}(2)  (2006), 666--678.

\bibitem {wic1} 
C. Wickham, {\em Classification of rings with genus one zero-divisor graphs}, {\it Communications in Algebra}  {\bf 36}(2) (2008),  325--345.

\bibitem {wic2} 
C. Wickham, {\em Rings whose zero-divisor graphs have positive genus},  Journal of Algebra, \textbf{321}(2) (2009),  377--383.

\bibitem{west}
D. B. West, {\em Introduction to Graph Theory} (Second  Edition), PHI Learning Private Limited, New Delhi, 2009.

\bibitem{whit}
Arthur T. White, {\em Graphs, Groups and Surfaces}, North-Holland Mathematics Studies, no. 8., American Elsevier Publishing Co., Inc.,  New York, 1973.

\bibitem{gap}
The GAP~Group, \emph{GAP -- Groups, Algorithms, and Programming, Version 4.6.4}, 2013.
  http://www.gap-system.org.

\end{thebibliography}
\end{document}